\documentclass[11pt,reqno]{amsart}
\usepackage{amsmath,amsopn,amssymb,amsthm}
\usepackage[cp1251]{inputenc}
\usepackage[english]{babel}
\usepackage[pagebackref,breaklinks=true,colorlinks=true,linkcolor=blue,citecolor=blue,urlcolor=blue]{hyperref}
\usepackage{graphicx}%

\newtheorem{theorem}{Theorem}[section]

\newtheorem{corollary}[theorem]{Corollary}

\newtheorem{lemma}[theorem]{Lemma}

\newtheorem{question}[theorem]{Question}

\newtheorem{remark}[theorem]{Remark}

\renewenvironment{proof}[1][Proof]{\noindent\textbf{#1.} }{\ \rule{0.5em}{0.5em}}

\begin{document}
\title[Some geometric correspondences for homothetic navigation]{Some geometric correspondences \\ for homothetic navigation}
\author{Ming Xu, Vladimir Matveev, Ke Yan, Shaoxiang Zhang }

\address{Ming Xu \newline
School of Mathematical Sciences,
Capital Normal University,
Beijing 100048,
P. R. China}
\email{mgmgmgxu@163.com}

\address{Vladimir Matveev\newline
Institut f\"{u}r Mathematik,
Fakult\"{a}t f\"{u}r Mathematik und Informatik,
Friedrich-Schiller-Universit\"{a}t Jena,
Germany
}
\email{vladimir.matveev@uni-jena.de}

\address{ Ke Yan\newline
School of Mathematical Sciences and LPMC,
Nankai University,
Tianjin 300071,
P.R. China
}
\email{1120160015@mail.nankai.edu.cn}

\address{ Shaoxiang Zhang,
the correspondence author\newline
School of Mathematical Sciences and LPMC,
Nankai University,
Tianjin 300071,
P.R. China
}
\email{zhangshaoxiang93@163.com}

\date{}
\maketitle

\begin{abstract}
In this paper, we provide conceptional explanations for the geodesic and Jacobi field correspondences for homothetic navigation, and then let them guide us to the shortcuts to some well known flag curvature and S-curvature formulas. They also help us directly see the local correspondence between isoparametric functions or isoparametric hypersurfaces,
which generalizes the classification works of Q. He and her coworkers for isoparametric
hypersurfaces in Randers space forms and Funk spaces.


\end{abstract}

\section{Introduction}

{\it Zermero navigation} (or {\it navigation} for simplicity) is an important technique which helps us produce new Finsler metrics and study their geometric properties.
The simplest non-Riemannian Finsler metrics, Randers metrics, can be produced by navigation from Riemannian metrics (see Subsection 5.4.2 in \cite{SS2016}).
If we use a homothetic vector field in the navigation datum to
produce the new metric,
we simply call this procedure a {\it homothetic navigation}.
{\it Killing navigation}, which uses a Killing vector field in the navigation datum,
provides an important subclass of homothetic navigation.
Homothetic navigation and Killing navigation are crucial for classifying Randers metrics of constant flag or Ricci curvature
\cite{BR2004,BRS2004} and
studying closed geodesics \cite{Ka1973,Zi1982}
in Finsler geometry.

Comparing the geometry before and after a homothetic navigation,
we see many similar features and beautiful correspondences.
For example, X. Mo and L. Huang proved their flag curvature
formula for homothetic navigation \cite{MH2007}.
\begin{theorem}\label{main-thm-1}
Let $\tilde{F}$ be the Finsler metric on $M$ defined by navigation from the datum $(F,V)$, in which $V$ is a homothetic vector field
with dilation $c$, then we have the equality between flag curvatures,
\begin{equation}\label{0040}
K^{\tilde{F}}(x,\tilde{y},\tilde{\mathbf{P}})
=K^{{F}}(x,y,{\mathbf{P}})-c^2.
\end{equation}
Here $x$ is any point with $F(x,-V(x))<1$, $y$ is any nonzero vector in $T_xM$,
the tangent plane $\mathbf{P}$ is spanned by $y$ and $u\in T_xM$ satisfying $\langle u,y\rangle_y^F=0$, and the tangent plane $\tilde{\mathbf{P}}$ is spanned by $u$ and $\tilde{y}=y+F(x,y)V(x)$.
\end{theorem}

The notions of homothetic vector field and its dilation
are according to the convention of Subsection 5.4.2 in \cite{SS2016}.
See Section 3 for equivalent definitions for them.
The equality (\ref{0040}) when $c=0$, i.e., the flag curvature formula for Killing navigation, was firstly found by
P. Foulon \cite{Fo2004}. There are many interesting applications
\cite{HD2012,XD2018} for his formula.

Recently,
Q. He and her coworkers classified locally all isoparametric hypersurfaces in a Randers space form $(M,F)$,
with respect to the Busemann-Hausdorff (or B.H. for simplicity) volume form $d\mu^F_{\mathrm{BH}}$ \cite{HDY2019}. Their classification result can be summarized as the following theorem.

\begin{theorem}\label{main-thm-5}
Let $\tilde{F}$ be a Randers metric defined by navigation from the datum
$(F,V)$, in which $F$ is a Riemannian metric with constant curvature and $V$ is a homothetic or Killing vector field for the
metric $F$. Then locally around any $x_0\in M$ where $F(x_0,-V(x_0))<1$,
any hypersurface is isoparametric for $(F,d\mu^F_{\mathrm{BH}})$
if and only if it is isoparametric for $(\tilde{F},d\mu^{\tilde{F}}_{\mathrm{BH}})$.
\end{theorem}

According to the work \cite{BRS2004} of D. Bao, C. Robles and Z. Shen, any Randers space form, i.e., Randers manifold with constant flag curvature, can be produced by homothetic or Killing navigation from a Riemannian space form.
In Riemannian geometry, any local isoparametric hypersurface in
a complete space form $M$ can be extended to a global one in the universal cover of $M$. When $M$ is an Euclidean space or a hyperbolic space, its global
isoparametric hypersurfaces  are classified by E. Cartan \cite{Ca1938}. When $M$ is a unit sphere, the classification
was recently completed by Q. Chi \cite{Ch2016} (see also the surveys \cite{QT2014,Th2000} and the references therein).
\medskip

Many proofs in the literature on Zermero navigation, for example, those for Theorem \ref{main-thm-1} in \cite{MH2007} and Theorem \ref{main-thm-5} in \cite{HDY2019},
have involved some
sophistical notions or complicated calculations. But we believe that
there must exist more straightforward explanations and easier proofs.

In a recent paper \cite{FM2017}, P. Foulon and the second author showed
a simple proof for P. Foulon's flag curvature formula for Killing navigation.
Their method inspired us to study the case of homothetic navigation, and see how some geometrical properties can be naturally fitted into
a system of correspondences. Firstly, we have a conceptional
explanation for the geodesic correspondence, reproving Theorem \ref{thm-1} which firstly appeared in \cite{HM2011}. Then as a corollary, we get the correspondence for orthogonal Jacobi fields
(see Theorem \ref{thm-2}). Since flag curvature can be described by Jacobi fields (see Lemma \ref{lemma-flag-curv-Jacobi}), we can use the
above correspondences to propose an alternative proof for
Theorem \ref{main-thm-1}, with a crystal theme and
minimized core calculation (see Lemma \ref{lemma-key-calculation}). By almost the same argument, we can even prove Theorem \ref{main-thm-1} when $F$ is pseudo-Finsler (i.e., Theorem 1.3 in \cite{JV2018}; see Remark \ref{remark-2}).

As a byproduct, similar thought and Lemma \ref{lemma-key-calculation} help us prove
\begin{theorem}\label{main-thm-6}
Let $\tilde{F}$ be the Finsler metric defined by navigation from the datum $(F,V)$, in which $V$ is a homothetic vector field with
dilation $c$. Then for the metrics
$F$ and $\tilde{F}$, and their B.H. volume forms
$d\mu^F_{\mathrm{BH}}$ and $d\mu^{\tilde{F}}_{\mathrm{BH}}$
respectively, we have the following equality between the S-curvatures $S^F(x,y)$ and $S^{\tilde{F}}(x,\tilde{y})$,
$$S^{\tilde{F}}(x,\tilde{y})=S^F(x,y)+(n+1)c,$$
in which $x\in M$ satisfies $F(x,-V(x))<1$,
$y$ is any $F$-unit vector in $T_xM$ and
$\tilde{y}=y+V(x)$.
\end{theorem}

Theorem \ref{main-thm-6} seems known in folklore.
Its special case
when $F$ is Riemannian is included in Theorem 5.10 in \cite{SS2016}.

As an application of Theorem \ref{main-thm-1},
We discuss
\begin{question} \label{question}
When can the locally symmetric property of the Finsler metric $F$ be preserved after a homothetic navigation?
\end{question}

Theorem 2 in \cite{FM2017} answers Question \ref{question} for Killing navigation, which always
preserves the locally symmetric property. Here we answer Question \ref{question} for non-Killing homothetic navigation, which only preserves the
locally symmetric property for flat metrics (see Theorem \ref{thm-application}).

Finally, we study the correspondence between normalized isoparametric functions (or isoparametric hypersurfaces) before and after a homothetic navigation. The notion of normalized isoparametric function implies that its gradient vector field
generates unit speed geodesics, for which we already have the correspondence by Theorem \ref{thm-1}. Comparing the B.H. volume forms and applying fundamental properties of Lie derivative, we can easily prove a relation between the Laplacians (see Lemma \ref{lemma-last}).
Now the local correspondence
between normalized isoparametric functions is obvious (see Theorem \ref{main-thm-3}), and we can generalize Theorem \ref{main-thm-5} to the following

\begin{theorem}\label{main-thm-4}
Let $V$ be a homothetic or Killing vector field on the Finsler manifold $(M,F)$, and $\tilde{F}$ the metric defined by navigation from the datum $(F,V)$. Then locally around any point $x_0$ with
$F(x_0,-V(x_0))<1$, a hypersurface is isoparametric for  $(F,d\mu^F_{\mathrm{BH}})$ if and only if it is isoparametric for
$(\tilde{F},d\mu^{\tilde{F}}_{\mathrm{BH}})$.
\end{theorem}

Our approach is more direct than that in \cite{HDY2019,HYS2016,HYS2017}, which studied the submanifold geometry in the Finsler context. Besides classifying isoparametric hypersurfaces in Randers space forms, Theorem \ref{main-thm-3} and Theorem \ref{main-thm-4} also help us understand
the intrinsic relation between the classification works in \cite{HYS2016} and
\cite{HYS2017}, for isoparametric hypersurfaces in Minkowski spaces and Funk spaces respectively. Furthermore, they provide
abundant examples of the isoparametric hypersurfaces in Finsler geometry (see the remark at the end of the paper).

For simplicity, we will mainly discuss
non-Killing homothetic navigation in this paper. With very minor
changes, all the statements for lemmas and theorems, and all the
arguments can be transplanted to the easier case of Killing navigation.
\medskip

In Section 2, we summarize some necessary knowledge on Finsler geometry. In Section 3,
we introduce the notions of homothetic vector field and navigation process. In Section 4, we discuss geodesic or Jacobi field correspondences with conceptional proofs. In Section 5, we prove Theorem \ref{main-thm-1} and Theorem \ref{main-thm-6}.
In Section 6, we apply Theorem \ref{main-thm-1} to answer Question \ref{question}. In Section 7, we study the local correspondence between isoparametric functions and prove Theorem \ref{main-thm-4}.

\section{Preliminaries}

In this section, we summarize some basic knowledge on Finsler geometry. See \cite{BCS2000,Sh2001,SS2016} for more details.
Throughout this paper,  we assume $M$ to be
a smooth manifold which real dimension is $n>0$.

A {\it Finsler metric} on $M$ is a continuous
function $F:TM\rightarrow [0,+\infty)$ which satisfies the following conditions for any {\it standard local coordinates} $x=(x^i)\in M$ and $y=y^i\partial_{x^i}\in T_xM$:
\begin{enumerate}
\item The restriction of $F$ to $TM\backslash 0$ is a positive smooth function.
\item For any $\lambda\geq0$, $F(x,\lambda y)=\lambda F(x,y)$.
\item When $y\neq0$, the Hessian matrix
$(g^F_{ij}(x,y))=(\frac12[F^2(x,y)]_{y^iy^j})$
is positive definite.
\end{enumerate}
We will call $(M,F)$ a {\it Finsler manifold}.
The restriction of $F$ to each tangent space $T_xM$ is called a {\it Minkowski norm}.

The Hessian matrix $(g^F_{ij}(x,y))$ defines an inner product
on $T_xM$, i.e.,
\begin{equation}\label{0007}
\langle u,v\rangle_y^F= g^F_{ij}(x,y)u^iv^j=
\frac12[F^2(x,y+su+tv)]_{st}|_{s=t=0},
\end{equation}
which depends on the choice of the nonzero base tangent vector $y$.
Sometimes we simply denote it as $g^F_y$ and call it the
{\it fundamental tensor}. The fundamental tensor $g^F_y$ is independent of the choice of $y$ in each tangent space if and only if $(M,F)$ is Riemannian.
\medskip

Arc length and distance can be similarly defined on the Finsler manifold $(M,F)$. A geodesic $\gamma(t)$ with $t\in(a,b)$
is a smooth nonconstant curve which satisfies the locally minimizing principle, i.e., for any $t_0\in(a,b)$, we can find $t_1$ and $t_2$ with $a<t_1<t_0<t_2<b$, such that $\gamma(t)$ with $t\in[t_1,t_2]$ is the unique curve realizing the distance from $x_1=\gamma(t_1)$ to $x_2=\gamma(t_2)$ \cite{BCS2000}.

We usually parametrize the geodesic $\gamma(t)$ such that $F(\dot{\gamma}(t))\equiv\mathrm{const}>0$ (or $F(\dot{\gamma}(t))\equiv1$), and call it a {\it constant speed geodesic} (or {\it unit speed geodesic} respectively). Constant speed geodesic can be equivalently defined by
the equation $D^F_{\dot{\gamma}(t)}\dot{\gamma}(t)\equiv0$.
Here the covariant derivative $D^F_{\dot{\gamma}(t)}$ is an ordinary
differential operator acting on the space of smooth vector fields
along $\gamma(t)$. See Section 5.3 in \cite{Sh2001} for its explicit expression. We will need the following property of covariant derivative.

\begin{lemma}\label{lemma-covariant-derivative}
For any smooth vector fields $U(t)$ and $V(t)$ along the  geodesic $\gamma(t)$ (i.e., $U(t),V(t)\in T_{\gamma(t)}M$ for all $t$, same below) on the Finsler manifold $(M,F)$, we have
\begin{equation}\label{0005}
\frac{d}{dt}\langle U(t),V(t)\rangle^F_{\dot{\gamma}(t)}
=\langle D^F_{\dot{\gamma}(t)}U(t),V(t)\rangle^F_{\dot{\gamma}(t)}
+\langle U(t),D^F_{\dot{\gamma}(t)}V(t)\rangle^F_{\dot{\gamma}(t)}.
\end{equation}
\end{lemma}

To be self contained, we sketch a short proof of Lemma \ref{lemma-covariant-derivative} here. We can extend $\dot{\gamma}(t)$ to a smooth vector field $Y$ in a neighborhood
$\mathcal{U}$ of $\gamma$, such that each integration curve of
$Y$ is a constant speed geodesic. The fundamental tensors $g^F_{Y}$
defines a smooth Riemannian metric on $\mathcal{U}$. The covariant derivative along $\gamma(t)$ for the Levi-Civita connection of $g^F_{Y}$ coincides
with $D^F_{\dot{\gamma}(t)}$ (see Lemma 6.2.1 in \cite{Sh2001}).
So we only need to observe (\ref{0005})
in Riemannian geometry, which is a well known fact.
\medskip

{\it Flag curvature} is a natural generalization of sectional curvature
in Riemannian geometry. For any $x\in M$, $y\in T_xM$, and tangent plane $\mathbf{P}\subset T_xM$ containing $y$,
the flag curvature $K^F(x,y,\mathbf{P})$ is defined by
$$K^F(x,y,\mathbf{P})=\frac{\langle u,R^F_y u\rangle_y^F}{
\langle y,y\rangle_y^F\langle u,u\rangle_y^F-(\langle y,u\rangle_y^F)^2},$$
where $u$ is any vector in $\mathbf{P}$ such that $\mathbf{P}=\mathrm{span}\{y,u\}$.
Here the linear operator $R^F_y:T_xM\rightarrow T_xM$ is the Riemann curvature (see \cite{BCS2000,Sh2001} for its explicit formula).

We call a smooth vector field $J(t)$ along
 the unit speed geodesic $\gamma(t)$ a {\it Jacobi field} if it
satisfies the Jacobi equation
\begin{equation*}
D^F_{\dot{\gamma}(t)}D^F_{\dot{\gamma}(t)}J(t)+
R^F_{\dot{\gamma}}J(t)=0.
\end{equation*}

For example, the variation of a smooth family of constant speed geodesics provides a Jacobi field
along each geodesic in this family. Conversely, any Jacobi field
can be locally realized in this way (see Lemma \ref{lemma-Jacobi-realization} for the special case we will use later).

We call the Jacobi field $J(t)$ {\it orthogonal},
if $J(t)$ is contained in the $g^F_{\dot{\gamma}(t)}$-orthogonal complement of $\dot{\gamma}(t)$, i.e.,
$\langle J(t),\dot{\gamma}(t)\rangle^F_{\dot{\gamma}(t)}=0$, for each value of $t$.
\medskip

Busemann-Hausedorff (B.H. in short) volume form on the Finsler manifold $(M,F)$ can be locally presented as $d\mu^F_{\mathrm{BH}}=\sigma^F dx^1\cdots dx^n$. Here $$\sigma^F=\frac{\mathrm{Vol}(S^n(1))}{
\mathrm{Vol}(\{y=(y^i)|F(x,y^i \partial_{x^i})\leq 1\})},$$
in which $\mathrm{Vol}(\cdot)$ denotes the volume with respect to
the standard measure in an Euclidean space.

For all standard local coordinates,
$$\tau^F(x,y)=\ln\left(\frac{\sqrt{\det(g_{ij}^F(x,y))}}{\sigma^F}\right)$$
globally defines a smooth function on $TM\backslash 0$, called
the {\it distortion function}. The {\it S-curvature}
$S^F(x,y)$ is defined as the derivative of
$\tau^F(x,y)$ in the direction of the geodesic spray, or equivalently, the derivative of $\frac{d}{dt}\tau^F(\gamma(t),\dot{\gamma}(t))|_{t=0}$,
in which $\gamma(t)$ is the constant speed geodesic $\gamma(t)$ on $(M,F)$, satisfying $\gamma(0)=x$ and $\dot{\gamma}(0)=y$.
\section{Homothetic vector field and Zermero navigation}
Let $V$ be a smooth vector field on the Finsler manifold $(M,F)$.
Around each $x\in M$, $V$ generates
a family (a one-parameter local subgroup) of local diffeomorphisms $\Psi_t$. We call $V$ a {\it homothetic vector field} on $(M,F)$ if
\begin{equation}\label{0006}
(\Psi_t^*F)(x,y)=F(\Psi_t(x),(\Psi_t)_*(y))= e^{-2ct} F(x,y),
\end{equation}
for each $x\in M$, $y\in T_xM$ and $t\in\mathbb{R}$, whenever $\Psi_t(x)$ is well defined. The constant $c$ in (\ref{0006}) is called the {\it dilation} of $V$. Notice that (\ref{0006}) indicates $\Psi_t$
are local homothetic translations. By (\ref{0006}) and (\ref{0007}), it is easy to see that, whenever $y\in T_xM$ is nonzero and $\Psi_t(x)$ is well defined, we have
\begin{equation}\label{0011}
\langle (\Psi_t)_*u,(\Psi_t)_*v\rangle_{(\Psi_t)_*y}^F
=e^{-4ct}\langle u,v\rangle^F_y,\quad\forall u,v\in T_xM.
\end{equation}
The homothetic vector field $V$ is a {\it Killing vector field} if its dilation $c$ is zero.

Since the local homothetic or isometric translations $\Psi_t$ maps
constant speed geodesics to constant speed geodesics, the restriction of the homothetic or Killing vector field
$V$ to any constant speed geodesic $\gamma(t)$ is a Jacobi field.
So $\langle V(\gamma(t)),\dot{\gamma}(t)\rangle_{\dot{\gamma}(t)}^F$ is a linear function. More precise information is given by
the following lemma.

\begin{lemma}\label{lemma-homo-Jacobi}
Let $V$ be a homothetic vector field with dilation $c$. Then its restriction to a unit
speed geodesic $\gamma(t)$ satisfies
\begin{equation}\label{0008}
\langle V(\gamma(t)),\dot{\gamma}(t)
\rangle^F_{\dot{\gamma}(t)}\equiv c_0-2ct,
\end{equation}
in which $c_0$ is some real constant.
\end{lemma}

\begin{proof}
The lemma is obvious when $V$ is constantly zero.

When $V$ is not constantly zero, we first prove this lemma locally
where $V$ is not tangent to $\gamma(t)$.
We can find local coordinates $x=(x^1,x')=(x^i)\in M$ and $y=y^i\partial_{x^i}$, such that $\gamma(t)=(0,t,0,\ldots,0)$ and $V$ coincides with $\partial_{x^1}$. Since $V$ is a homothetic vector field
with dilation $c$, the metric $F$ can be presented
as $F(x,y)=e^{-2cx^1}F_1(x',y)$.

From the assumption that
$\gamma(t)$ is a unit speed geodesic, i.e.,
$D^F_{\dot{\gamma}(t)}{\dot{\gamma}(t)}=0$, we can get
$$\frac{\partial}{\partial x^2}[F^2(x,\partial_{x^2})]_{y^1}
=[F^2(x,\partial_{x^2})]_{x^2 y^1}=[F^2(x,\partial_{x^2})]_{x^1}=-4c,$$
when $x^1=x^3=\cdots=x^n=0$.
Solving this differential equation with respect to the variable $x^2$, we see
$\langle V(\gamma(t)),\dot{\gamma}(t)\rangle_{\dot{\gamma}(t)}^F
=\frac{1}{2}[F^2(x,\partial_{x^2})]_{y^1}|_{(0,t,0,\ldots,0)}$
is a linear function of $t$ which slope is $-2c$.

By continuity, Lemma \ref{lemma-homo-Jacobi} is valid everywhere
for all unit speed geodesics.
\end{proof}

By similar argument with local coordinates, it is also easy to see
\begin{lemma}\label{lemma-homo-BH}
Let $V$ be a homothetic vector field with dilation $c$ on the Finsler manifold
$(M^n,F)$, and ${\Psi}_t$ the local homothetic translation generated by
$V$. Then we have the equality for the B.H. volume forms
\begin{equation*}
(\Psi_t)^*d\mu^F_{\mathrm{BH}} =d\mu^{\Psi_t^* F}_{\mathrm{BH}}=e^{-2cnt}d\mu^F_{\mathrm{BH}},
\end{equation*}
whenever $\Psi_t$ is well defined.
\end{lemma}

Assume that $V$ is a smooth vector field satisfying $F(x,-V(x))<1$ in some open subset $\mathcal{U}\subset M$. Then the equality
$F(x,y)=\tilde{F}(x,\tilde{y})$
defines a new Finsler metric on $\mathcal{U}$,
in which $\tilde{y}=y+F(x,y)V(x)$ for any $x\in\mathcal{U}$ and $y\in T_xM$
(see Section 5.4 in \cite{SS2016}).
We will call $\tilde{F}$ {\it the metric defined by
navigation}  {\it from the datum} $(F,V)$.

A relation between the fundamental tensors of $F$ and $\tilde{F}$
is revealed by the following lemma (see Lemma 4.4 in \cite{Xu2018} or the equality (5) in \cite{FM2017}).

\begin{lemma}\label{lemma-navigation-fundamental-tensor}
Let $\tilde{F}$ be the Finsler metric on $M$ defined by navigation
from the datum $(F,V)$, then for any $x\in \mathcal{U}$ and nonzero vector $y\in T_xM$, we have
$$\langle u,v\rangle_{\tilde{y}}^{\tilde{F}}=\frac{1}{1+\langle V(x),y\rangle_y^F}
\langle u,v\rangle_y^F,$$
for any $u$ and $v$ in the $g^F_y$-orthogonal complement of $y$ in $T_xM$.
\end{lemma}

In each tangent space $T_xM$ for $x\in\mathcal{U}$, the indicatrix
$S^{\tilde{F}}_xM=\{y\in T_xM| \tilde{F}(y)=1\}$
is a parallel shifting of the indicatrx $S^F_xM$ by the
vector $V(x)$, so it is easy to prove (see Proposition 5.3 in \cite{SS2016})
\begin{lemma}\label{lemma-navi-BH-volume}
Let $\tilde{F}$ be the Finsler metric on $M$ defined by navigation
from the datum $(F,V)$. Then $d\mu^F_{\mathrm{BH}}=d\mu^{\tilde{F}}_{\mathrm{BH}}$
inside $\mathcal{U}$.
\end{lemma}

\section{Geodesic and Jacobi field correspondences}
\label{sec-4}
Unless otherwise specified, we keep the following setup for the rest of the paper.
Let $V$ is a homothetic vector field on
the Finsler manifold $(M,F)$ with dilation $c\neq0$.
We will fix a point $x_0\in M$ with $F(x_0,-V(x_0))<1$ and restrict our discussion to a sufficiently small neighborhood $\mathcal{U}$ of $x_0$ where the metric $\tilde{F}$ can be defined
in $\mathcal{U}$ by navigation from the datum $(F,V)$. The parameter $t$ for a unit speed geodesic passing $x\in\mathcal{U}$ when $t=0$, or
for the local homothetic translations $\Psi_t$ generated by $V$,
is understood to be sufficiently close to zero.

In \cite{HM2011}, L. Huang and X. Mo proved the following correspondence between
unit speed geodesics before and after a homothetic navigation.

\begin{theorem}\label{thm-1}
For any unit speed
geodesic $\gamma(t)$ for the metric $F$ with
$\gamma(0)=x\in \mathcal{U}$,
$\tilde{\gamma}(t)=\Psi_t(\gamma(\frac{e^{2ct}-1}{2c}))$ is a unit speed geodesic for the metric $\tilde{F}$ with $\tilde{\gamma}(0)=x$.
Conversely,
any unit speed geodesic $\tilde{\gamma}(t)$ for the metric $\tilde{F}$ with $\tilde{c}(0)=x$
can be presented in this way.
\end{theorem}

Following a similar thought as in \cite{FM2017}, we propose a conceptional proof of it.

\begin{proof}
Firstly we assume $\gamma(t)$ is a unit speed geodesic on $(M,F)$ and prove  $\tilde{\gamma}(t)$ is a unit speed geodesic for the metric $\tilde{F}$.

Direct calculation shows
$$F(\gamma(\frac{e^{2ct}-1}{2c}),
\frac{d}{dt}{\gamma}(\frac{e^{2ct}-1}{2c}))=
F(\gamma(\frac{e^{2ct}-1}{2c}),e^{2ct}\dot{\gamma}(
\frac{e^{2ct}-1}{2c}))
=e^{2ct}.$$
By (\ref{0006}), we have
\begin{eqnarray*}
& &F(\Psi_t(\gamma(\frac{e^{2ct}-1}{2c})),(\Psi_t)_*
(\frac{d}{dt}{\gamma}(\frac{e^{2ct}-1}{2c})))=
 e^{-2ct}
F(\gamma(\frac{e^{2ct}-1}{2c}),
\frac{d}{dt}{\gamma}(\frac{e^{2ct}-1}{2c}))\\
& =&e^{-2ct}F(\gamma(\frac{e^{2ct}-1}{2c}),
e^{2ct}\dot{\gamma}(\frac{e^{2ct}-1}{2c}))=
F(\gamma(\frac{e^{2ct}-1}{2c}),
\dot{\gamma}(\frac{e^{2ct}-1}{2c})=1,
\end{eqnarray*}
and by the notion of navigation,
$$\dot{\tilde{\gamma}}(t)=
(\Psi_t)_*
(\frac{d}{dt}{\gamma}(\frac{e^{2ct}-1}{2c}))
+V(\tilde{\gamma}(t))$$
is a $\tilde{F}$-unit tangent vector.
To summarize,
$\tilde{\gamma}(t)$ is
a $\tilde{F}$-unit speed curve.

Assume conversely that $\tilde{\gamma}(t)$ is
not a geodesic, i.e., the local minimizing principle is not valid
somewhere on $\tilde{\gamma}(t)$, then we can find a pair of real numbers $t'$ and $t''$, with $t'<t''$ sufficiently close to each other, and satisfying the following:
\begin{enumerate}
\item the segment of $\gamma(t)$ with $t\in[\frac{e^{2ct'}-1}{2c},\frac{e^{2ct''}-1}{2c}]$ is the unique minimizing geodesic between its end points for the metric $F$.
\item $\tilde{\gamma}(t)$ with $t\in[t',t'']$ is not minimizing
for the metric $\tilde{F}$.
\end{enumerate}

Because of (2),
we can find another $\tilde{F}$-unit speed smooth curve $\tilde{\gamma}_1(t)$
    such that it
    coincides with $\tilde{\gamma}(t)$ when $t\notin(t',t'')$.
From $\tilde{\gamma}_1(t)$, we can trace back to find an
 $F$-unit speed smooth curve
$\gamma_1(t)$, such that $\tilde{\gamma}_1(t)=\Psi_t(\gamma_1(\frac{e^{2ct}-1}{2c}))$. The curve $\gamma_1(t)$ is different from $\gamma(t)$. But both coincide  when $t\notin(\frac{e^{2ct'}-1}{2c},\frac{e^{2ct''}-1}{2c})$ and
have the same $F$-length
$\frac{1}{2c}(e^{2ct''}-e^{2ct'})$
for the segment $t\in[\frac{e^{2ct'}-1}{2c},\frac{e^{2ct''}-1}{2c}]$.
This is a contradiction to (1). So $\tilde{\gamma}(t)$ must be a geodesic
for the metric $\tilde{F}$.

This proves the first statement in Theorem \ref{thm-1}.

To prove the other statement in Theorem \ref{thm-1}, we observe that at $\tilde{\gamma}(0)=\gamma(0)=x$,
$\dot{\tilde{\gamma}}(0)=\dot{\gamma}(0)+V(x)$ can exhaust all
$\tilde{F}$-unit tangent vectors. All other arguments are similar.
\end{proof}
\medskip

Using Theorem \ref{thm-1}, we can prove a correpondence between the orthogonal Jacobi fields for the metrics $F$ and $\tilde{F}$ respectively.

\begin{theorem}\label{thm-2}
For any orthogonal Jacobi field $J(t)$ along the unit speed geodesic $\gamma(t)$ for the metric $F$,
\begin{equation}\label{0002}
\tilde{J}(t)=(\Psi_t)_* (J(\frac{e^{2ct}-1}{2c}))
\end{equation}
defines an orthogonal Jacobi
field along the unit speed geodesic $\tilde{\gamma}(t)=\Psi_t(\gamma(\frac{e^{2ct}-1}{2c}))$ for the metric $\tilde{F}$.
Conversely, any orthogonal
Jacobi field $\tilde{J}(t)$ along $\tilde{\gamma}(t)$ for the metric $\tilde{F}$
can be presented by (\ref{0002}) for some orthogonal
Jacobi field $J(t)$ along $\gamma(t)$ for the metric $F$.
\end{theorem}

\begin{proof}
Firstly, we assume $J(t)$ is an orthogonal Jacobi field along the unit speed geodesic $\gamma(t)$ for the metric $F$ and prove $\tilde{J}(t)$ is
an orthogonal Jacobi field along $\tilde{\gamma}(t)$ for the metric $\tilde{F}$.

The orthogonal property of $J(t)$ can be equivalently described as
the claim that
$J(\frac{e^{2ct}-1}{2c})$ is tangent to the indicatrix of ${F}$ in
$T_{\gamma(\frac{e^{2ct}-1}{2c})}M$ at $\dot{\gamma}(\frac{e^{2ct}-1}{2c})$.
Since $\Psi_t$ is a local
homothetic translation, $\tilde{J}(t)=(\Psi_t)_*(J(\frac{e^{2ct}-1}{2c}))$
is also tangent to the indicatrix of $F$ in $T_{\tilde{\gamma}(t)}M$
at $(\Psi_t)_*(\frac{d}{dt}\gamma(\frac{e^{2ct}-1}{2c}))$.
Since the indicatrix of $\tilde{F}$ is a parallel shifting of that of $F$ by the value of $V$, $\tilde{J}(t)$
is tangent to the indicatrix of $\tilde{F}$ in $T_{\tilde{\gamma}(t)}M$ at the $\tilde{F}$-unit vector
$\dot{\tilde{\gamma}}(t)=(\Psi_t)_*(\frac{d}{dt}\gamma(\frac{e^{2ct}-1}{2c}))
+V(\tilde{\gamma}(t))$ as well.
To summarize, we have proved the orthogonal property for $\tilde{J}(t)$. Then we will prove $\tilde{J}(t)$ is a Jacobi
field for the metric $\tilde{F}$.

By Lemma \ref{lemma-Jacobi-realization} below, $J(t)$ can be realized as $J(t)=\frac{\partial}{\partial s}\gamma(s,t)|_{s=0}$ for a smooth variation $\gamma(s,t)$ of $\gamma(t)=\gamma(0,t)$, such that for each $s$, $\gamma(s,t)$ is a unit speed geodesic for the metric $F$. Then we have
$\tilde{J}(t)=\frac{\partial}{\partial s}\tilde{\gamma}(s,t)|_{s=0}$, where
$\tilde{\gamma}(s,t)=\Psi_t(\gamma(s,\frac{e^{2ct}-1}{2c}))$. By Theorem \ref{thm-1},
$\tilde{\gamma}(s,t)$ is a smooth variation of $\tilde{\gamma}(t)=\tilde{\gamma}(0,t)$, such that for each $s$, $\tilde{\gamma}(s,t)$ is a unit speed geodesic for the metric $\tilde{F}$. So $\tilde{J}(t)$ is a Jacobi field along $\tilde{\gamma}(t)$ for the metric $\tilde{F}$.

This argument proves the first statement in Theorem \ref{thm-2}.
The proof for the second statement is similar.
\end{proof}

\begin{lemma}\label{lemma-Jacobi-realization}
For any orthogonal Jacobi field $J(t)$ along the unit speed geodesic $\gamma(t)$ for the metric $F$ satisfying $\gamma(0)=x$, we can find
a smooth map $\gamma(s,t)$ with $s\in(-\epsilon,\epsilon)$ such that for each fixed $s$, $\gamma(s,t)$ is a unit speed geodesic for the metric $F$, $\gamma(0,t)=\gamma(t)$ and $\frac{\partial}{\partial s}\gamma(0,t)=J(t)$.
\end{lemma}
\begin{proof}
Let $t_1$ and $t_2$ with $t_1<0<t_2$ be real numbers which are sufficiently close to 0, and $\gamma_i(s)$ smooth curves
with $\gamma_i(0)=\gamma(t_i)$ and $\dot{\gamma}_i(0)=J(t_i)$, for $i=1$ and $2$ respectively. For each fixed $s$ sufficiently close to $0$, there exists a unique unit speed geodesic $\gamma(s,t)$ from $\gamma_1(s)$ to $\gamma_2(s)$, which can be suitably extended and parametrized such that $\gamma(s,t_1)=\gamma_1(s)$.

We only need to prove that $J(t)$ coincides with
the Jacobi field $\bar{J}(t)=\frac{\partial}{\partial s}\gamma(0,t)$. For each fixed $s$, we denote $l(s)=d_F(\gamma_1(s),\gamma_2(s))$ the distance from $\gamma_1(s)$ to $\gamma_2(s)$.
By the orthogonal property of $J(t)$, i.e., $$\langle J(t_1),\dot{\gamma}(t_1)\rangle_{\dot{\gamma}(t_1)}^F
=\langle J(t_2),\dot{\gamma}(t_2)\rangle^F_{\dot{\gamma}(t_2)}=0,$$ the first variation indicates $\frac{d}{ds}l(0)=0$. So we have  $\bar{J}(t_2)=\dot{\gamma}_2(0)=J(t_2)$. Meanwhile we also have
$\bar{J}(t_1)=\dot{\gamma}_1(0)=J(t_1)$. When $t_1$ and $t_2$ are sufficiently close to $0$, $\gamma(t_2)$ is not a conjugate point of $\gamma(t_1)$ along the geodesic $\gamma(t)$, i.e., the values at $t=t_1$ and $t_2$ uniquely determine the Jacobi field. So we must have $J(t)\equiv\bar{J}(t)$,
which ends the proof of this lemma.
\end{proof}

For any tangent vector $u$ in the $g^F_{\dot{\gamma}(0)}$-orthogonal complement of $\dot{\gamma}(0)$ at $\gamma(0)=x$,
we denote $\mathcal{J}^F_{\gamma;u}$ the set of all orthogonal
Jacobi fields $J(t)$ along $\gamma(t)$ for the metric $F$, satisfying $J(0)=u$. Then Theorem \ref{thm-2} immediately implies the following

\begin{corollary}\label{coro-thm-2}
The correspondence from $J(t)$ to $\tilde{J}(t)$ in Theorem \ref{thm-2}
is one-to-one between $\mathcal{J}^F_{\gamma;u}$
and $\mathcal{J}^{\tilde{F}}_{\tilde{\gamma};u}$.
\end{corollary}
\section{Proofs of Theorem \ref{main-thm-1} and Theorem \ref{main-thm-6}}
To prove Theorem \ref{main-thm-1}, we need the following description of flag curvature
by Jacobi fields.

\begin{lemma} \label{lemma-flag-curv-Jacobi}
Let $\gamma(t)$ be a unit speed geodesic for the metric $F$
with $\gamma(0)=x\in M$ and $\dot{\gamma}(0)=y\in T_xM$. Suppose
the tangent plane $\mathbf{P}\subset T_xM$ is spanned by $y$ and the nonzero vector $u$ satisfying $\langle u,y\rangle_y^F=0$.
Then we have
\begin{equation}\label{0010}
K^F(x,y,\mathbf{P})=(\langle u,u\rangle_y^F)^{-1/2}
\max_{J(t)\in\mathcal{J}^F_{\gamma;u}} \{-\frac{d^2}{dt^2}|_{t=0}[(\langle J(t),J(t)
\rangle_{\dot{\gamma}(t)}^F)^{1/2}]\},
\end{equation}
\end{lemma}

\begin{proof}
Let $J(t)$ be any orthogonal Jacobi field along the unit speed geodesic
$\gamma(t)$ for the metric $F$, satisfying $J(0)=u$.
Using Lemma \ref{lemma-covariant-derivative},
we can get
\begin{eqnarray*}
& &\frac{d^2}{dt^2}[(\langle J(t),J(t)\rangle_{\dot{\gamma}(t)}^F)^{1/2}]
=\frac{d}{dt}\left(\frac{\langle D^F_{\dot{\gamma}(t)}J(t),J(t)\rangle_{\dot{\gamma}(t)}^F}{
(\langle J(t),J(t)\rangle_{\dot{\gamma}(t)}^F)^{1/2}}
\right)\\
&=&\frac{
\langle D^F_{\dot{\gamma}(t)}J(t), D^F_{\dot{\gamma}(t)}J(t)\rangle_{\dot{\gamma}(t)}^F
\langle J(t),J(t)\rangle_{\dot{\gamma}(t)}^F-
(\langle D^F_{\dot{\gamma}(t)}J(t),J(t)\rangle_{\dot{\gamma}(t)}^F)^2}{
(\langle J(t),J(t)\rangle_{\dot{\gamma}(t)}^F)^{3/2}}\\
& &+\frac{\langle J(t),D^F_{\dot{\gamma}(t)}D^F_{\dot{\gamma}(t)}J(t)
\rangle_{\dot{\gamma}(t)}^F}{
(\langle J(t),J(t)\rangle_{\dot{\gamma}(t)}^F)^{1/2}}\\
&\geq&\frac{\langle J(t),D^F_{\dot{\gamma}(t)} D^F_{\dot{\gamma}(t)}J(t)\rangle_{\dot{\gamma}(t)}^F}{
(\langle J(t),J(t)\rangle_{\dot{\gamma}(t)}^F)^{1/2}}
=-\frac{\langle J(t),R^F_{\dot{\gamma}(t)}J(t)\rangle_{\dot{\gamma}(t)}^F}{
(\langle J(t),J(t)\rangle_{\dot{\gamma}(t)}^F)^{1/2}},
\end{eqnarray*}
in which we have used Cauchy inequality.
So at $t=0$, we have
\begin{eqnarray}
K^F(x,y,\mathbf{P})&=&\frac{\langle R^F_{\dot{\gamma}(0)}J(0),J(0)\rangle_{\dot{\gamma}(0)}^F}{
\langle J(0),J(0)\rangle_{\dot{\gamma}(0)}^F}\nonumber\\
&\geq&-(\langle u,u\rangle_{y}^F)^{-1/2}
\frac{d^2}{dt^2}[(\langle J(t),J(t)
\rangle_{\dot{\gamma}(t)}^F)^{1/2}]|_{t=0}
\label{0004}.
\end{eqnarray}

This calculation proves
\begin{equation}\label{0012}
K^F(x,y,\mathbf{P})\geq(\langle v,v\rangle_y^F)^{-1/2}
\max_{J(t)\in\mathcal{J}^F_{\gamma;u}} \{-\frac{d^2}{dt^2}(\langle J(t),J(t)
\rangle_{\dot{\gamma}(t)}^F|_{t=0})^{1/2}\}.
\end{equation}

Notice that there exists a unique Jacobi field $J(t)$ along $c(t)$ such that
$J(0)=u$ and $D^F_{\dot{\gamma}(t)}J(t)|_{t=0}=0$. This Jacobi field is orthogonal, i.e., $J(t)\in\mathcal{J}^F_{\gamma;u}$, because
$\langle J(t),\dot{\gamma}(t)\rangle^F_{\dot{\gamma}(t)}$ is a linear function of $t$, and when $t=0$, we have
$\langle J(0),\dot{\gamma}(0)\rangle^F_{\dot{\gamma}(0)}
=\langle u,y\rangle^F_{y}=0$
and by Lemma \ref{lemma-covariant-derivative},
\begin{eqnarray*}
\frac{d}{dt}|_{t=0}\langle J(t),\dot{\gamma}(t)\rangle^F_{\dot{\gamma}(t)}
=
\langle D^F_{\dot{\gamma}(t)}J(t)|_{t=0},\dot{\gamma}(0)
\rangle_{y}^F+
\langle J(0),D^F_{\dot{\gamma}(t)}\dot{\gamma}(t)|_{t=0}
\rangle_{y}^F=0.
\end{eqnarray*}

From previous calculation, it is easy to see that the equality and maximum in (\ref{0012}) is achieved simultaneously by this $J(t)$. This ends the proof of Lemma \ref{lemma-flag-curv-Jacobi}.
\end{proof}
\medskip

The most crucial calculation for a homothetic navigation is
contained in the following lemma.

\begin{lemma}\label{lemma-key-calculation}
Let $v$ be a tangent vector at ${\gamma}(\frac{e^{2ct}-1}{2c})$
 in the $g^F_{\dot{\gamma}(\frac{e^{2ct}-1}{2c})}$-orthogonal complement of
 $\dot{\gamma}(\frac{e^{2ct}-1}{2c})$. Then
 $(\Psi_t)_*(v)$  is a tangent vector at $\Psi_t({\gamma}(\frac{e^{2ct}-1}{2c}))$ in
 the $g^{\tilde{F}}_{\dot{\tilde{\gamma}}(t)}$-orthogonal complement
 of $\dot{\tilde{\gamma}}(t)$, which satisfies
 \begin{equation}\label{0100}
 \langle(\Psi_t)_*(v),(\Psi_t)_*(v)
 \rangle_{\dot{\tilde{\gamma}}(t)}^{\tilde{F}}
 = \frac{1}{c_0+1}\cdot e^{-2ct}\langle v,v
 \rangle^F_{\dot{\gamma}(\frac{e^{2ct}-1}{2c})},
 \end{equation}
 in which $c_0$ is the constant in Lemma \ref{lemma-homo-Jacobi}.
\end{lemma}

\begin{proof}
Firstly, the argument in the proof of Theorem \ref{thm-1} covers the first statement of Lemma \ref{lemma-key-calculation},
i.e., $\langle (\Psi_t)_*(v),\dot{\tilde{\gamma}}(t)
\rangle^{\tilde{F}}_{\dot{\tilde{\gamma}}(t)}=0$. So we only need to verify (\ref{0100}).

Denote $y(t)=\frac{d}{dt}\gamma(\frac{e^{2ct}-1}{2c})$ and
$\bar{y}(t)=\dot{\tilde{\gamma}}(t)-V(\tilde{\gamma}(t))
=(\Psi_t)_*(y(t)).$
By Lemma \ref{lemma-navigation-fundamental-tensor},
\begin{eqnarray}\label{0030}
\langle (\Psi_t)_*(v),(\Psi_t)_*(v)\rangle^{\tilde{F}}_{\bar{y}(t)}=
\frac{1}{1+\langle
V(\tilde{\gamma}(t)),
\bar{y}(t)
\rangle^F_{\bar{y}(t)}}
\langle (\Psi_t)_*(v),(\Psi_t)_*(v)\rangle^F_{
\bar{y}(t)}.
\end{eqnarray}
By (\ref{0011}), we get
\begin{eqnarray}\label{0031}
\langle (\Psi_t)_*(v),(\Psi_t)_*(v)\rangle^F_{
\bar{y}(t)}=
\langle (\Psi_t)_*(v),(\Psi_t)_*(v)\rangle^F_{
(\Psi_t)_*({y}(t))}
= e^{-4ct}
\langle v,v
\rangle^F_{y(t)},
\end{eqnarray}
and
\begin{eqnarray}\label{0032}
\langle V(\tilde{\gamma}(t)),\bar{y}(t)\rangle^F_{\bar{y}(t)}&=&
\langle (\Psi_t)_* (V({\gamma}(\frac{e^{2ct}-1}{2c}))),(\Psi_t)_*(y(t))
\rangle^F_{(\Psi_t)_*(y(t))}
\nonumber\\
&=& e^{-4ct}\langle V(\frac{e^{2ct}-1}{2c}),
y(t)\rangle^F_{
\dot{\gamma}(\frac{e^{2ct}-1}{2c})}\nonumber\\
&=& e^{-2ct}\langle V(\frac{e^{2ct}-1}{2c}),
\dot{\gamma}(\frac{e^{2ct}-1}{2c})\rangle^F_{
\dot{\gamma}(\frac{e^{2ct}-1}{2c})}.
\end{eqnarray}
By Lemma \ref{lemma-homo-Jacobi},
\begin{eqnarray}\label{0033}
\langle V(\frac{e^{2ct}-1}{2c}),
\dot{\gamma}(\frac{e^{2ct}-1}{2c})\rangle^F_{
\dot{\gamma}(\frac{e^{2ct}-1}{2c})}=c_0-2c
\cdot\frac{e^{2ct}-1}{2c}=
(c_0+1)-e^{2ct},
\end{eqnarray}
in which $c_0$ is the constant in Lemma \ref{lemma-homo-Jacobi}.

Summarizing (\ref{0030})-(\ref{0033}), we get
\begin{eqnarray*}
\langle (\Psi_t)_*v,(\Psi_t)_*v\rangle^{\tilde{F}}_{\dot{\tilde{\gamma}}(t)}
&=&\frac{1}{1+e^{-2ct}((c_0+1)-e^{2ct})}\cdot
e^{-4ct}\langle v,v\rangle^F_{\dot{\gamma}(\frac{e^{2ct}-1}{2c})}
\\
&=&\frac1{c_0+1}\cdot e^{-2ct} \langle
v,v\rangle^F_{\dot{\gamma}(\frac{e^{2ct}-1}{2c})}.
\end{eqnarray*}

This ends the proof of Lemma \ref{lemma-key-calculation}.
\end{proof}
\medskip

Now we summarize all the observations in these two sections
to prove Theorem \ref{main-thm-1}.

\begin{proof}[Proof of Theorem \ref{main-thm-1}]
Firstly,
we assume the dilation $c$ is not zero. For any fixed $x\in M$, we only need to restrict our
discussion locally in a suitable open neighborhood of $x$.
Let $y$ be any $F$-unit tangent vector in $T_xM$, then there exists
a unique unit speed $\gamma(t)$ for the metric $F$, such that
$\gamma(0)=x$ and $\dot{\gamma}(0)=y$. By Theorem \ref{thm-1},
$\tilde{\gamma}(t)=\Psi_t(\gamma(\frac{e^{2ct}-1}{2c}))$ is a unit speed geodesic
for the metric $\tilde{F}$, satisfying $\tilde{\gamma}(0)=x$
and $\dot{\tilde{\gamma}}(0)=\tilde{y}=y+V(x)$.

For any orthogonal Jacobi field $J(t)$
in $\mathcal{J}^F_{\gamma;u}$, we denote
$f_J(t)=(\langle J(t),J(t)\rangle_{\dot{\gamma}(t)}^F)^{1/2}$.
By Lemma \ref{lemma-flag-curv-Jacobi}, we have
\begin{equation}\label{0101}
K^F(x,y,\mathbf{P})=(\langle u,u\rangle_y^F)^{-1/2}\max_{J(t)\in \mathcal{J}^F_{\gamma;u}}\{-{f_J}''(0)\}.
\end{equation}

By Theorem \ref{thm-2}, $\tilde{J}(t)=(\Psi_t)_*(J(\frac{e^{2ct}-1}{2c}))
\in\mathcal{J}^{\tilde{F}}_{\tilde{\gamma};u}$, i.e., it is an orthogonal
Jacobi field along $\tilde{\gamma}(t)$ for the metric $\tilde{F}$
satisfying $\tilde{J}(0)=u$.
We denote $\tilde{f}_J(t)=\langle \tilde{J}(t),\tilde{J}(t)\rangle_{\dot{\gamma}(t)}^{\tilde{F}}$.
Then
by Corollary \ref{coro-thm-2} and Lemma \ref{lemma-flag-curv-Jacobi},
\begin{equation}\label{0102}
K^{\tilde{F}}(x,\tilde{y},\tilde{\mathbf{P}})
=(\langle u,u\rangle_{\tilde{y}}^{\tilde{F}})^{-1/2}
\max_{{J}(t)\in\mathcal{J}^{{F}}_{{\gamma};u}}
\{-\tilde{f}_J{}''(0)\}.
\end{equation}

By Lemma \ref{lemma-key-calculation}, for the same
$J(t)\in \mathcal{J}^F_{\gamma;u}$, we have
\begin{equation}\label{0400}
\tilde{f}_J(t)\equiv\sqrt{\frac{1}{1+c_0}}
e^{-ct}f_J(\frac{e^{2ct}-1}{2c}),
\end{equation}
where $c_0$ is the constant in Lemma \ref{lemma-homo-Jacobi}.
Evaluate (\ref{0400}) at $t=0$, we can determine the constant
$$\sqrt{\frac{1}{1+c_0}}
=\frac{\tilde{f}_J(0)}{f_J(0)}=\frac{(\langle u,u\rangle_{\tilde{y}}^{\tilde{F}})^{1/2}}{(\langle u,u\rangle_y^F)^{1/2}}.$$

It is easy to calculate that
\begin{eqnarray}
{\tilde{f}_J}''(0)&=& \frac{(\langle u,u\rangle_{\tilde{y}}^{\tilde{F}})^{1/2}}{(\langle u,u\rangle_y^F)^{1/2}}\cdot \frac{d^2}{dt^2}(e^{-ct}f_J(\frac{e^{2ct}-1}{2c}))|_{t=0}\nonumber\\
&=& \frac{(\langle u,u\rangle_{\tilde{y}}^{\tilde{F}})^{1/2}}{(\langle u,u\rangle_y^F)^{1/2}}\cdot({c^2} f_J(0)+{f_J}''(0)).\label{0103}
\end{eqnarray}
Finally, summarizing (\ref{0101}), (\ref{0102}) and (\ref{0103}), we get
\begin{eqnarray*}
K^{\tilde{F}}(x,\tilde{y},\tilde{\mathbf{P}})&=&
(\langle u,u\rangle_{\tilde{y}}^{\tilde{F}})^{-1/2}
\max_{J(t)\in\mathcal{J}^F_{\gamma;u}}
\{-\tilde{f}_J{}''(0)\}\\
&=&(\langle u,u\rangle_{{y}}^{{F}})^{-1/2}
\max_{J(t)\in\mathcal{J}^F_{\gamma;u}}
\{-({c^2} f_J(0)+{f_J}''(0))\}\\
&=&(\langle u,u\rangle_y^F)^{-1/2}\max_{J(t)\in\mathcal{J}^F_{\gamma;u}}
\{-f_J{}''(0)\}-{c^2}\\
&=& K^F(x,y,\mathbf{P})-{c^2}.
\end{eqnarray*}

This ends the proof of Theorem \ref{main-thm-1} when $c\neq0$.
With only some minor modifications, this argument can also prove
the case $c=0$, which ends the proof of Theorem \ref{main-thm-1}. \end{proof}

\begin{remark} In \cite{FM2017}, Theorem \ref{main-thm-1}
with $c=0$ is proved by a
slightly different approach.
\end{remark}

The authors of \cite{FM2017} used another description for the flag curvature, i.e.,
\begin{equation}\label{another-description-flag-curvature}
K^F(x,y,\mathbf{P})=\frac{1}{2\langle u,u\rangle_y^F}
\max_{J(t)\in \mathcal{J}^F_{\gamma;u}}
\{-\frac{d^2}{dt^2}|_{t=0}({f_J}^2(t))\},
\end{equation}
to prove Theorem \ref{main-thm-1} for Killing navigation.
The maximum in (\ref{another-description-flag-curvature})
can only be achieved when $D^F_{\dot{\gamma}(t)}J(t)|_{t=0}=0$.
Implied by Lemma \ref{lemma-key-calculation}, $D^{\tilde{F}}_{\dot{\tilde{\gamma}}(t)}\tilde{J}(t)|_{t=0}\neq0$ when $c\neq0$.
So their proof can not be directly generalized to
 homothetic navigation.

\begin{remark} \label{remark-2}
Our proof of Theorem \ref{main-thm-1} can be easily generalized to pseudo-Finsler geometry.
\end{remark}

In \cite{JV2018}, M. Javaloyes and H. Vit\'{o}rio proved Theorem \ref{main-thm-1} when $F$ is pseudo-Finsler. Their proof applied the fanning curves approach \cite{Ah1989}.

Our method can also be applied to prove their theorem (i.e., Theorem 1.3 in \cite{JV2018}). The geodesic correspondence is similar to the Finsler case (see Theorem 1.2 in \cite{JV2018}). From
the view point of variation, the expected correspondence between orthogonal Jacobi fields follows immediately.
Lemma \ref{lemma-key-calculation} with the key calculation can be
proved by the same argument.

When we use orthogonal Jacobi fields
to describe flag curvature, we can not use Lemma \ref{lemma-flag-curv-Jacobi} directly. The reason is the following.
When $F$ is pseudo-Finsler, the fundamental tensor $\langle\cdot,\cdot\rangle_y^F$ may be indefinite, and then Cauchy inequality used in the proof of Lemma \ref{lemma-flag-curv-Jacobi} fails.
However, since the restriction of $\langle\cdot,\cdot\rangle_y^F$ to $\mathbf{P}=\mathrm{span}\{y,u\}$ is nondegenerate and $\langle u,y\rangle_y^F=0$, we must have $\langle u,u\rangle^F_y\neq0$. By similar calculation as in the proof of Lemma \ref{lemma-flag-curv-Jacobi}, we can show that $K^F(x,y,\mathbf{P})$ is the unique critical value
of the functional
$$L(J)=-|\langle u,u\rangle_y^F|^{1/2}
\frac{d^2}{dt^2}|_{t=0}(|\langle J(t),J(t)\rangle_{
\dot{\gamma}(t)}^F|^{1/2}),\quad\forall J(t)\in\mathcal{J}^F_{\gamma;u},$$ and
the critical set is the affine subspace of all
$J(t)\in\mathcal{J}^F_{\gamma;u}$ such that
$D^F_{\dot{\gamma}(t)}J(t)|_{t=0}$ is a scalar multiple of $u$. So
we can still use the calculation (\ref{0103}) to prove the flag curvature equality (\ref{0040}).
\medskip

The geodesic correspondence for homothetic navigation and the key calculation in Lemma \ref{lemma-key-calculation} also help us prove Theorem \ref{main-thm-6}.

\begin{proof}[Proof of Theorem \ref{main-thm-6}]
Let $\gamma(t)$ be the unit speed geodesic om $(M,F)$, satisfying $\gamma(0)=x$ and $\dot{\gamma}(0)=y$, and denote $\tilde{\gamma}(t)=\Psi_t(\gamma(\frac{e^{2ct}-1}{2c}))$. Firstly,
we choose smooth vector fields $e_i(t)$, $1\leq i\leq n$, along the geodesic $\gamma(t)$, such that $e_1(t)=\dot{\gamma}(t)$, and
they provide a $g^F_{\dot{\gamma}}$-orthonormal basis of $T_{\gamma(t)}M$ for each  $t$.
Secondly, we define the following smooth vector fields along $\tilde{\gamma}(t)$,
\begin{eqnarray*}
\bar{e}_i(t)&=&(\Psi_t)_*(e_i(\frac{e^{2ct}-1}{2c})), \mbox{ for } 1\leq i\leq n,\\
\tilde{e}_1(t)&=&\dot{\tilde{\gamma}}(t), \mbox{ and}\\
\tilde{e}_i(t)&=&\bar{e}_i(t)
\mbox{ for } 1<i\leq n.
\end{eqnarray*}
Then at each point $\tilde{\gamma}(t)$, $\{\bar{e}_i(t)\mbox{ with } 1\leq i\leq n\}$ and $\{\tilde{e}_i(t)\mbox{ with }1\leq i\leq n\}$ are two bases for $T_{\tilde{\gamma}(t)}M$.
Denote
\begin{eqnarray*}
\mathrm{vol}(t)&=&\mathrm{Vol}(\{(y^i)|F(\gamma(t),y^i e_i(t))\leq1\}),\\
\overline{\mathrm{vol}}(t)&=&\mathrm{Vol}
(\{(y^i)|F(\tilde{\gamma}(t),
y^i\bar{e}_i(t))\leq1\}),\mbox{ and}\\
\widetilde{\mathrm{vol}}(t)&=&\mathrm{Vol}(\{
(y^i)|F(\tilde{\gamma}(t),
y^i\tilde{e}_i(t))\leq1\}),
\end{eqnarray*}
in which $\mathrm{Vol}$ is the standard measure in
an Euclidean space.

Using this setup, the distortions $\tau^F(\gamma(t),\dot{\gamma}(t))$ and
$\tau^{\tilde{F}}(\tilde{\gamma}(t),\dot{\tilde{\gamma}}(t))$, for the metric $F$ and $\tilde{F}$ respectively, can be
presented as
\begin{eqnarray}
\tau^F(\gamma(t),\dot{\gamma}(t))&=&
\ln\sqrt{\det(\langle e_i(t),e_j(t)\rangle^F_{\dot{\gamma}(t)})}-\ln\mathrm{vol}(t)+C_0,
\mbox{ and}\label{tau-1}\\
\tau^{\tilde{F}}(\tilde{\gamma}(t),\dot{\tilde{\gamma}}(t))
&=&\ln\sqrt{\det(\langle\tilde{e}_i(t),\tilde{e}_j(t)
\rangle^{\tilde{F}}_{\dot{\tilde{\gamma}}(t)})}+
\ln\widetilde{\mathrm{vol}}(t)+C_0,
\label{tau-2}
\end{eqnarray}
in which $C_0$ is some universal constant depending on $n$.

By Lemma \ref{lemma-key-calculation},
\begin{equation}\label{0060}
\det(\langle\tilde{e}_i(t),\tilde{e}_j(t)
\rangle_{\dot{\tilde{\gamma}}(t)})=
C\cdot e^{-2c(n-1)t}
\det(\langle e_i(\frac{e^{2ct}-1}{2c}),e_j(\frac{e^{2ct}-1}{2c})
\rangle^F_{\dot{\gamma}(t)}),
\end{equation}
in which $C={1}/({1+c_0})^{n-1}$ is some positive constant.
By (\ref{0032}), (\ref{0033}), the homothetic property of $V$ and multi-variable calculus,
\begin{eqnarray}\label{0061}
\widetilde{\mathrm{vol}}(t)&=&\frac{
\overline{\mathrm{vol}}(t)
}{1+\langle
V(\Psi_t(\gamma(\frac{e^{2ct}-1}{2c}))),\bar{e}_1(t)
\rangle^F_{\bar{e}_1(t)}}
=C' e^{2ct}\overline{\mathrm{vol}}(t)\nonumber\\
&=&C' e^{2ct}\cdot e^{2c(n-1)t}
\mathrm{vol}(\frac{e^{2ct}-1}{2c})= C' e^{2cnt}\mathrm{vol}(\frac{e^{2ct}-1}{2c})
,
\end{eqnarray}
in which $C'$ is some positive constant.

Summarizing (\ref{tau-1})-(\ref{0061}), we get
$$\tau^{\tilde{F}}(\tilde{\gamma}(t),\dot{\tilde{\gamma}}(t))
=\tau^F(\gamma(\frac{e^{2ct}-1}{2c}),\dot{\gamma}
(\frac{e^{2ct}-1}{2c}))
+c(n+1)t,
$$
so
\begin{eqnarray*}
S^{\tilde{F}}(x,\tilde{y})&=&\frac{d}{dt}
\tau^{\tilde{F}}(\tilde{\gamma}(t),\dot{\tilde{\gamma}}(t))|_{t=0}\\
&=&
(\frac{e^{2ct}-1}{2c})'|_{t=0}\cdot
\frac{d}{ds}\tau^{F}(\gamma(s),\dot{\gamma}(s))|_{s=0}+(n+1)c\\
&=&
S^{F}(x,y)+(n+1)c.
\end{eqnarray*}

This ends the proof of Theorem \ref{main-thm-6} when $c\neq0$.
The case $c=0$ can be proved similarly.
\end{proof}

\section{Application to the study of locally symmetric property}

In this section, we discuss an application of Theorem \ref{main-thm-1}.

Recall that a Finsler metric $F$ is called {\it locally symmetric} ({\it in curvature sense}) if for any unit speed geodesic
$\gamma(t)$, we have $D^F_{\dot{\gamma}(t)}R^F_{\dot{\gamma}(t)}\equiv0$
\cite{Fo1997}.

This is a weaker version for the notion of locally symmetric Finsler metric, compared to the one defined in metric sense, i.e.,
around each point $x$, we can find a local involutive isometry $\rho_x$, such that $x$ is an isolated fixed point of $\rho_x$.
Notice that locally symmetric Finsler metric in metric sense must be Berwaldian \cite{MT2012}, but there are many non-Berwaldian Finsler spheres with constant curvature, which are locally symmetric.

In \cite{FM2017}, the second author and P. Foulon proved that
the locally symmetric property is preserved by Killing navigation.
However, for non-Killing homothetic navigation, the following theorem indicates a very different phenomenon.

\begin{theorem}\label{thm-application}
Let $\tilde{F}$ be the
metric defined by navigation from the datum $(F,V)$, in which $V$ is a homothetic vector field on $M$ satisfying $F(x,-V(x))<1$
in an open subset $\mathcal{U}$ and its dilation $c$ is nonzero. Then the following two
statements are equivalent:
\begin{enumerate}
\item The metric $F$ has constant zero flag curvature in $\mathcal{U}$.
\item Both $F$ and $\tilde{F}$ are locally symmetric in $\mathcal{U}$.
\end{enumerate}
\end{theorem}

\begin{proof}
Firstly, we prove the statement from (1) to (2). By Theorem \ref{main-thm-1}, the flag curvature of the metric $\tilde{F}$
is constantly $-c^2$ in $\mathcal{U}$. As Finsler metrics of constant flag curvature are locally symmetric, both $F$ and $\tilde{F}$ are
locally symmetric in $\mathcal{U}$, which proves the statement from (1) to (2).

Nextly, we prove the statement from (2) to (1). For any $x\in \mathcal{U}$ and $F$-unit vector $y\in T_xM$, we denote
\begin{eqnarray*}
K^F_{\max}(x,y)&=&\max\{K^F(x,y,\mathbf{P})|\forall\mbox{ tangent plane }\mathbf{P}\mbox{ with }y\in\mathbf{P}\},\mbox{ and}\\
K^F_{\min}(x,y)&=&\min\{K^F(x,y,\mathbf{P})|\forall\mbox{ tangent plane }\mathbf{P}\mbox{ with }y\in\mathbf{P}\}.
\end{eqnarray*}

We claim $K^F_{\max}(x,y)\equiv K^F_{\min}(x,y)\equiv0$.

The proof for $\lambda=K^F_{\max}(x,y)=0$ is as following. We choose a unit speed geodesic $\gamma(t)$ in $\mathcal{U}$ for the metric $F$, such that $\gamma(0)=x$ and
$\dot{\gamma}(0)=y$. The locally symmetric property of $F$ implies
that $K^F_{\max}(\gamma(t),\dot{\gamma}(t))$ is a constant function of $t$, i.e.,
$K^F_{\max}(\gamma(t),\dot{\gamma}(t))\equiv 
K^F_{\max}(\gamma(0),\dot{\gamma}(0))=\lambda$. By Theorem \ref{thm-1}, $\tilde{\gamma}(t)=\Psi_t(\gamma(\frac{e^{2ct}-1}{2c}))$ is
a unit speed geodesic for the metric $\tilde{F}$, in which $\Psi_t$'s are the local homothetic translations generated by $V$. Denote $\bar{y}(t)=(\Psi_t)_*(\frac{d}{dt}\gamma(\frac{e^{2ct}-1}{2c}))$,
then the homothetic property implies
$$K^F_{\max}(\tilde{\gamma}(t),\bar{y}(t))=e^{4ct}
K^F_{\max}(\gamma(t),
\dot{\gamma}(t))=\lambda e^{4ct}.$$
Applying Theorem \ref{main-thm-1} to all tangent planes $\mathbf{P}$ and $\tilde{\mathbf{P}}$, containing $\bar{y}(t)$ and $\dot{\tilde{\gamma}}(t)$ respectively, we get
\begin{equation}
\label{1001}
K^{\tilde{F}}_{\max}(\tilde{\gamma}(t),
\dot{\tilde{\gamma}}(t))
=K^F_{\max}(\tilde{\gamma}(t),\bar{y}(t))-c^2=\lambda e^{4ct}-c^2.
\end{equation}
By the locally symmetric property of $\tilde{F}$, $K^{\tilde{F}}_{\max}(\tilde{\gamma}(t),
\dot{\tilde{\gamma}}(t))$ is a constant function of $t$. Since $c\neq0$, we must have
$K^F_{\max}(x,y)=\lambda=0$ for any $x\in M$ and $F$-unit vector
$y\in T_xM$.

Similarly, we can prove $K^F_{\min}(x,y)=0$ for any $x\in M$ and
$F$-unit vector $y\in T_xM$. Then it is obvious to see $K^F\equiv0$, and $K^{\tilde{F}}\equiv -c^2$ by Theorem \ref{main-thm-1}.
To summarize, this argument proves
the statement from (2) to (1).
\end{proof}

\section{Homothetic navigation for isoparametric function}

In this section, we keep all assumptions and
the notations for the Finsler manifold $(M,F)$ and the homothetic navigation as in Section \ref{sec-4}.
Further more, we consider some locally defined isoparametric
functions in some open subset $\mathcal{U}$ of $M$ where $F(x,-V(x))<1$
is satisfied.

Let $f $ be a regular function in $\mathcal{U}$. The notion of {\it isoparametric} property for $f $ is defined by the following conditions \cite{HYS2016}:
\begin{enumerate}
\item The $F$-length function $F(\nabla^F f )$ for gradient vector field $\nabla^F f$ only depends on the values of $f $, i.e., $f $ is {\it transnormal}.
\item The Laplacian $\Delta^F f $ only depends on the values of $f $.
\end{enumerate}

Here $\nabla^F$ is the {\it gradient operator} for the metric $F$, defined by
$\langle\nabla^F f ,W\rangle^F_{\nabla^F f }=df(W)$ for any vector field $W$.
The regularity of $f(x)$ implies $\nabla^F f(x)$ is well defined and $F(\cdot,\nabla^F f )$ is a positive smooth function in $\mathcal{U}$.

Denote $\mathrm{div}_{d\mu}$ the divergence operator with respect to
the smooth volume form $d\mu$, i.e.,
$\mathrm{div}_{d\mu} W\cdot d\mu=\mathcal{L}_W d\mu$,
in which $W$ is any smooth vector field, and $\mathcal{L}$ is
the Lie derivative.
Then the
{\it (nonlinear) Laplacian} $\Delta^F f $ can be presented as
$\Delta^F f =
\mathrm{div}_{\mu^F_{\mathrm{BH}}}\nabla^F f$.

We can always replace an isoparametric or transnormal function $f $ by $\varphi\circ f$ for
some suitable real smooth function $\varphi$, such that
\begin{equation}\label{normal-isoparametric}
F(\nabla^F f )\equiv 1,\mbox{ and }f(x_0)=0\mbox{ for some fixed }x_0\in\mathcal{U}.
\end{equation}
Notice that the isoparametric and transnormal properties are preserved, and the local foliation of the level sets is unchanged. So we only need to consider $f $ satisfying
(\ref{normal-isoparametric}), which is simply called {\it normalized around $x_0$}.\medskip

Now we consider the navigation with the datum $(F,V)$, in which $V$ is a homothetic vector field with dilation $c\neq0$. We study its effect
on the foliation $M_t=f^{-1}(t)$ locally defined by a normalized transnormal function $f(x)$ around each fixed $x_0\in M_0$. We
assume $F(x_0,-V(x_0))<1$, so that $\tilde{F}$ is well defined
around $x_0$. By Lemma 4.1 in \cite{Xu2018}, the integration curves of $\nabla^F f$ are unit speed geodesics on $(M,F)$.

We define a smooth map $\Psi$ locally around $x_0$, such that
$\Psi|_{M_{\frac{e^{2ct-1}}{2c}}}=\Psi_t$ for each value of $t$.

\begin{lemma}\label{lemma-foliation-navigation}
$\Psi$ is an orientation preserving local diffeomorphism around $M_0$.
\end{lemma}

\begin{proof} It is obvious that $\Psi$ fixes each point of $M_0$. We only need to prove that
for any $x\in M_0$, the tangent map $\Psi_*:T_{x}M\rightarrow T_xM$ is an orientation preserving linear isomorphism. Then Lemma \ref{lemma-foliation-navigation} is obvious by this observation.

The tangent map $\Psi_*$ maps $T_{x}M_0$ identically to itself.
Let $\gamma(t)$ be the unit speed geodesic for the metric $F$, such that $\dot{\gamma}(t)=\nabla^F f(\gamma(t))$ and $\gamma(0)=x$. Theorem \ref{thm-1} indicates that $\tilde{\gamma}(t)=\Psi(\gamma(t))$ is a
unit speed geodesic for the metric $\tilde{F}$ with
$\dot{\tilde{\gamma}}(0)=\Psi_*(\nabla^F f(x))=\nabla^F f(x)+V(x)$.
Since we have assumed $F(x,-V(x))<1$, the strong
convexity of $F$ implies
$\langle\nabla^F f(x )+V(x ),\nabla^F f(x )\rangle_{\nabla^F f(x )}^F>0$.
So $\Psi_*:T_xM\rightarrow T_xM$ is an orientation preserving linear isomorphism for each $x\in M_0$.
\end{proof}

By Lemma \ref{lemma-foliation-navigation}, we can define the smooth function $\tilde{f}$ locally around $x_0$, with the level sets $$\tilde{f}^{-1}(t)=\widetilde{M}_{t}
=\Psi(M_{\frac{e^{2ct}-1}{2c}})=\Psi_t(M_{\frac{e^{2ct}-1}{2c}}).$$

Let $\gamma(t)$ be any integration curve of $\nabla^F f$ with $\gamma(0)\in M_0$ sufficiently close to $x_0$ and $t$ sufficiently close to zero. Denote the points $x={\gamma}(\frac{e^{2ct}-1}{2c})\in M_{\frac{e^{2ct}-1}{2c}}$
and $\tilde{x}=\Psi(x)\in \widetilde{M}_t$. Notice that
$\tilde{f}(\tilde{x})=t$ and $f(x)=\frac{e^{2ct}-1}{2c}$.

Theorem \ref{thm-1}
provides a unit speed geodesic $$\tilde{\gamma}(t)=\Psi(\gamma(\frac{e^{2ct}-1}{2c}))=
\Psi_t(\gamma(\frac{e^{2ct-1}}{2c}))$$
for the metric $\tilde{F}$, so $$\dot{\tilde{\gamma}}(t)=\Psi_*(\frac{d}{dt}{\gamma}
(\frac{e^{2ct}-1}{2c}))=\Psi_*(e^{2ct}\dot{\gamma}
(\frac{e^{2ct}-1}{2c}))=
\Psi_*((2cf(x)+1)\nabla^F f(x))$$ is a $\tilde{F}$-unit vector which is $g^{\tilde{F}}_{\dot{\tilde{\gamma}}(t)}$-orthogonal
to $T_{\tilde{x}}\widetilde{M}_t$. This implies
$$\nabla^{\tilde{F}}\tilde{f}(\tilde{x})=\dot{\tilde{\gamma}}(t)
=\Psi_*((2cf(x)+1)\nabla^F f(x)).$$
To summarize, we have the following lemma.
\begin{lemma} \label{lemma-navigation-normal-direction}
If $f$ is a normalized transnormal function around $x_0$ for the metric $F$, then
$\tilde{f}$ is a normalized transnormal function around $x_0$
for the metric $\tilde{F}$, and
$\nabla^{\tilde{F}}\tilde{f}=\Psi_*((2cf+1)\nabla^F f)$.
\end{lemma}

 Comparing $d\mu^F_{\mathrm{BH}}(x)$
and $d\mu^{\tilde{F}}_{\mathrm{BH}}(\tilde{x})$, we get

\begin{lemma}\label{lemma-navigation-BH-2}
$\Psi^*(d\mu^{\tilde{F}}_{\mathrm{BH}}(x))
=(1+c_0(x))(2cf(x)+1)^{-n-1}d\mu^F_{\mathrm{BH}}(x),
$
in which $c_0(\cdot)$ is a smooth function around $x_0$ which
is constant along each integration curve of $\nabla^F f$.
\end{lemma}

\begin{proof} Firstly, 
we have $d\mu^{\tilde{F}}_{\mathrm{BH}}=d\mu^F_{\mathrm{BH}}$ by Lemma \ref{lemma-navi-BH-volume}, so we only need to concern the
metric $F$ in the following discussion.

We fix any value of $t$ and consider the tangent map $(\Phi_t)_*$ for $\Phi_t=\Psi_{-t}\circ\Psi$ at $x=\gamma(\frac{e^{2ct}-1}{2c})$. The
restriction of $(\Phi_t)_*$ to $T_{x}M_{\frac{e^{2ct}-1}{2c}}$ is
the identity map, and $(\Phi_t)_*$ maps $\nabla^F f(x)=\dot{\gamma}(\frac{e^{2ct}-1}{2c})$ to
$\dot{\gamma}(\frac{e^{2ct}-1}{2c})+ e^{-2ct}V(x)$. So by Lemma \ref{lemma-navigation-fundamental-tensor},
\begin{eqnarray*}
\Phi_t^*(d\mu^F_{\mathrm{BH}}(x))&=&\det((\Phi_t)_*|_{T_xM})\cdot
d\mu^F_{\mathrm{BH}}
\\
&=&
\langle \dot{\gamma}(\frac{e^{2ct}-1}{2c})+e^{-2ct}V(x),
\dot{\gamma}(\frac{e^{2ct}-1}{2c})
\rangle^F_{\dot{\gamma}(\frac{e^{2ct}-1}{2c})}\cdot
d\mu^F_{\mathrm{BH}}(x)\\
&=&
(1+e^{-2ct}\langle V(x),\dot{\gamma}(\frac{e^{2ct}-1}{2c})
\rangle^F_{\dot{\gamma}(\frac{e^{2ct}-1}{2c})})\cdot
d\mu^F_{\mathrm{BH}}(x)\\
&=&(1+e^{-2ct}(c_0 -2c\cdot\frac{e^{2ct}-1}{2c}))\cdot
d\mu^F_{\mathrm{BH}}(x)\\
&=& e^{-2ct}(1+c_0 )\cdot d\mu^F_\mathrm{BH},
\end{eqnarray*}
Here $\det((\Phi_t)_*|_{T_xM})$ is the determinant of the matrix
for $(\Phi_t)_*|_{T_xM}$ with respect to any $g^F_{\dot{\gamma}(\frac{e^{2ct}-1}{2c})}$-orthonormal basis
$\{e_1=\dot{\gamma}(\frac{e^{2ct}-1}{2c}),e_2,\ldots,e_n\}$.
The constant
$c_0 $ is provided by Lemma \ref{lemma-homo-Jacobi} which depends on the geodesic $\gamma(t)$. So locally around $x_0$, we
can denote it as a smooth function $c_0(x)$, which is constant
along each integration curve of $\nabla^F f$.

By Lemma \ref{lemma-homo-BH},
$\Psi_t^*(d\mu^F_{\mathrm{BH}}(\tilde{x}))=
e^{-2cnt}d\mu^F_{\mathrm{BH}}
(x)$, in which $\tilde{x}=\Psi(x)$.
So we have
\begin{eqnarray*}
\Psi^*(d\mu^F_{\mathrm{BH}}(\tilde{x}))
&=&\Phi_t^*\Psi_t^*(d\mu^F_{\mathrm{BH}}(\tilde{x}))
=\Phi_t^*(e^{-2cnt}d\mu^F_{\mathrm{BH}}
(x))\\
&=&e^{-2c(n+1)t}(1+c_0(x))\cdot d\mu^F_{\mathrm{BH}}(x)\\
&=&(1+c_0(x))(2cf(x)+1)^{-n-1}d\mu^F_{\mathrm{BH}}(x).
\end{eqnarray*}

This ends the proof of Lemma \ref{lemma-navigation-BH-2}.
\end{proof}

Summarizing above discussion, we can prove the following key lemma.

\begin{lemma}\label{lemma-last}
Let $f$ be a normalized transnormal function around $x_0$, then
\begin{equation}\label{0080}
\Psi^*\Delta^{\tilde{F}}\tilde{f}=(2cf+1)\Delta^F f -2cn.
\end{equation}
\end{lemma}

\begin{proof}
By Lemma \ref{lemma-navi-BH-volume}, Lemma \ref{lemma-navigation-normal-direction}  and Lemma \ref{lemma-navigation-BH-2},
\begin{eqnarray}
\Psi^*\Delta^{\tilde{F}}\tilde{f}
&=&\Psi^*\mathrm{div}_{d\mu^F_{\mathrm{BH}}}(\Psi_*((2cf+1)\nabla^F f))=\mathrm{div}_{\Psi^*d\mu^F_{\mathrm{BH}}}((2cf+1)\nabla^F f)\nonumber\\
&=&\mathrm{div}_{((1+c_0(x))(2cf+1)^{-n-1}
d\mu^F_\mathrm{BH})}
((2cf+1)\nabla^F f)\nonumber\\
&=&\frac{\mathcal{L}_{((2cf+1)\nabla^F f)}
((1+c_0(x))(2cf+1)^{-n-1}
d\mu^F_\mathrm{BH})}{(1+c_0(x))(2cf+1)^{-n-1}
d\mu^F_\mathrm{BH}}\nonumber\\
&=&(\nabla^F f)(2cf+1)+(2cf+1)
\cdot\frac{\mathcal{L}_{\nabla^F f}((1+c_0(x))(2cf+1)^{-n-1}
d\mu^F_\mathrm{BH})}{(1+c_0(x))(2cf+1)^{-n-1})
d\mu^F_\mathrm{BH}}\nonumber\\
&=&2c+(2cf+1)\Delta^F f +(2cf+1)
(\nabla^F f)(\ln(2cf+1)^{-n-1})+\ln(1+c_0(x)))\nonumber\\
&=&(2cf+1)\Delta^F f -2cn,
\label{0070}
\end{eqnarray}
in which $c_0(x)$ does not appear in the last line because $(\nabla^F f)(c_0(x))\equiv0$ by
Lemma \ref{lemma-navigation-BH-2}.
\end{proof}


Obviously, when $f$ is isoparametric for $(F,d\mu^F_{\mathrm{BH}})$,  $\Delta^F f$ is constant on each $M_t$. By Lemma \ref{lemma-last}, $\Delta^{\tilde{F}}\tilde{f}$ is constant on each $\widetilde{M}_t$,
i.e., $\tilde{f}$ is a normalized isoparametric function around $x_0$ for $(\tilde{F},d\mu^{\tilde{F}}_{\mathrm{BH}})$.

To summarize, we have proved
\begin{theorem}\label{main-thm-2}
Let $\tilde{F}$ be the Finsler metric defined by navigation from
the datum $(F,V)$ in which $V$ is a homothetic vector field with
dilation $c\neq0$. Assume $x_0$ is a point where $F(x_0,-V(x_0))<1$. Then for any
normalized isoparametric function $f$ for  $(F,d\mu^F_{\mathrm{BH}})$ around the point $x_0$,
the function $\tilde{f}$ defined by
$\tilde{f}^{-1}(t)=\Psi_t(f^{-1}(\frac{e^{2ct}-1}{2c}))$
is a normalized isoparametric function
for $(\tilde{F},d\mu^{\tilde{F}}_{\mathrm{BH}})$ around $x_0$.
\end{theorem}

Nextly, we consider a normalized isoparametric function $\tilde{f}$ for $(\tilde{F},d\mu^{\tilde{F}}_{\mathrm{BH}})$
around $x_0\in M$. We can construct a smooth function $f$ locally around $x_0$,
such that $\tilde{f}^{-1}(t)=\Psi_t(f^{-1}(\frac{e^{2ct}-1}{2c}))$. By similar argument as for Lemma \ref{lemma-navigation-normal-direction}, we can prove $f$ is
a normalized transnormal function around $x_0$. Using Lemma \ref{lemma-last} again, it is easy to see that when $\Delta^{\tilde{F}}\tilde{f}$ is constant on each level set of
$\tilde{f}$, $\Delta^F f$ is constant on each level set of $f$, i.e., $f$ is a normalized isoparametric function for $(F,d\mu^F_{\mathrm{BH}})$.
So Theorem \ref{main-thm-2} can be strengthened as following.

\begin{theorem}\label{main-thm-3}
Keep all assumptions and notation in Theorem \ref{main-thm-2}. Then we have a one-to-one correspondence from $f$ to $\tilde{f}$
between normalized isoparametric functions around $x_0$, with respect to $(F,d\mu^F_{\mathrm{BH}})$ and $(\tilde{F},
d\mu^{\tilde{F}}_{BH})$ respectively.
\end{theorem}

Above argument also works in the case that $V$ is a Killing vector field. In this case, we only need to modify $\Psi$ such that $\Psi|_{M_t}=\Psi_t$, and make a few more minor changes accordingly. The correspondence between normalized isoparametric functions around $x_0$ is then from $f$ to $\tilde{f}=(\Psi^{-1})^* f$. To avoid iteration, we skip the details.
\medskip

\begin{proof}[Proof of Theorem \ref{main-thm-4}]
For any isoparametric hypersurface $N$ for either $F$ or $\tilde{F}$, locally around $x_0$ where $F(x_0,-V(x_0))<1$, we can find a normalized isoparametric function accordingly, such that $N$ is the level set for the zero value. By Theorem \ref{main-thm-3} and its Killing navigation version, $N$ is also isoparametric for the other metric.
\end{proof}
\medskip

%
%
%
%

Finally, we remark that Theorem \ref{main-thm-4} helps us find abundant examples of non-homogeneous isoparametric hypersurfaces
in Finsler geometry.

Let $(M,F)$ be a Finsler manifold admitting the cohomogeneity one isometric action of a connected Lie group $G$, such that each $G$-orbit is closed in $M$. Then principal $G$-orbits are  homogeneous
isoparametric hypersurface for $(F,d\mu^F_{\mathrm{BH}})$ \cite{Xu2018}.
Denote $\tilde{F}$ the metric
defined by navigation from the datum $(F,V)$ in which $V$
is a Killing or homothetic vector field. Then the non-empty
intersection between any principal $G$-orbit $G\cdot x$ and $\mathcal{U}=\{x\in M| F(x,-V(x))<1\}$
provides isoparametric hypersurfaces for $(\tilde{F},d\mu^{\tilde{F}}_{\mathrm{BH}})$.
Generally speaking, the connected isometry group of $(M,\tilde{F})$ is
smaller than that of $(M,F)$. So very likely,
many homogeneous isoparametric hypersurfaces for $(F,d\mu^F_{\mathrm{BH}})$ lose their homogeneity
after the navigation. See Theorem 5.4 in \cite{Xu2018} for the case that $(M,\tilde{F})$ is a Randers sphere of constant flag curvature.
\medskip

{\bf Acknowledgements.} The first author is supported by National Natural Science Foundation of China (No.~11821101, No.~11771331), Beijing Natural Science Foundation
(No.~00719210010001, No.~1182006), Research Cooperation Contract (No.~Z180004), and Capacity Building for Sci-Tech  Innovation -- Fundamental Scientific Research Funds (No.~KM201910028021). The second author is support by DFG
(projects MA 2565/4 and  MA 2565/6). He
would also like to thank Capital Normal University in Beijing China for hospitality during the preparation of this paper. All the authors would like to thank Qun He sincerely for her precious suggestions.

\vspace{5mm}
\end{document}